\documentclass{amsart}
\usepackage{amsmath,amsthm,amssymb,hyperref}
\usepackage[utf8]{inputenc}
\usepackage[english]{babel}
\usepackage[T1]{fontenc}
\usepackage[babel=true]{microtype}
\usepackage[margin=3.5cm]{geometry}
\usepackage{enumerate}

\numberwithin{equation}{section}

\usepackage[backend=biber,style=alphabetic,giveninits]{biblatex}
\usepackage[dvipsnames]{xcolor}
\definecolor{mylightgray}{gray}{0.8}
\definecolor{mygray}{gray}{0.6}
\definecolor{mydarkgray}{gray}{0.4}
\definecolor{myverydarkgray}{gray}{0.2}

\usepackage{ytableau}

\usepackage[backgroundcolor = white]{todonotes}

\theoremstyle{plain}
\newtheorem{theorem}{Theorem}[section]
\newtheorem{lemma}[theorem]{Lemma}

\newtheorem{proposition}[theorem]{Proposition}

\theoremstyle{definition}
\newtheorem{definition}[theorem]{Definition}
\newtheorem{example}[theorem]{Example}
\newtheorem{notation}[theorem]{Notation}

\theoremstyle{remark}

\newtheorem{remark}[theorem]{Remark}
\newtheorem*{acknowledgements}{Acknowledgements}

\newcommand{\sym}[1][]{\ifstrempty{#1}{
		\ensuremath{S(\infty)}
	}{
		\ensuremath{S_{#1}(\infty)}
	}
}
\newcommand\rest[2]{\ensuremath{\left.{#1}\right|_{#2}}} 

\newcommand{\G}{\mathfrak{G}}
\newcommand{\g}{\mathfrak{g}}
\newcommand{\T}{\mathfrak{T}}

\DeclareMathOperator{\dom}{dom}
\DeclareMathOperator{\range}{range}

\DeclareMathOperator{\supp}{supp}
\DeclareMathOperator{\End}{End}

\DeclareMathOperator{\Ind}{Ind}
\DeclareMathOperator{\id}{id}
\DeclareMathOperator{\PB}{PB}
\DeclareMathOperator{\IrrAdm}{IrrAdm}
\DeclareMathOperator{\Adm}{Adm}
\DeclareMathOperator{\spanset}{span}

\title[Semigroup approach to admissible representations of \( S(\infty) \)]{Semigroup approach to admissible representations of \\ the infinite symmetric group}
\author{Irina Devyatkova}
\address{
	\newline
	Igor Krichever Center for Advanced Studies, Skolkovo Institute of Science and Technology,  Moscow, Russia
	\newline
HSE University, Moscow, Russia}
\email{irina.e.devyatkova@gmail.com}

\begin{document}

\maketitle

\begin{abstract}
	Let \( S(\infty) \) denote the group of finitary permutations of the set \(\mathbb N:=\{1,2,3,\dots\}\). It is a countable group admitting a lot of different topologies compatible with the group structure. In particular, such topologies arise from partitions of the set \(\mathbb N\) into blocks of infinite size. The corresponding categories of continuous unitary representations of \(S(\infty)\) were studied by Nessonov (Sbornik: Mathematics, 2012). We propose a different approach to his classification results based on the so-called semigroup method. Some additional information is also obtained.
\end{abstract}

\section{Introduction}
\subsection{}

Suppose that we have an infinite-dimensional group \( G \) (usually \( G \) is an inductive limit of some compact, finite or finite-dimensional groups \( G(n) \)) and suppose that we fix some subgroup \( K \subset G \) which would play the role of a maximal compact subgroup of \( G \). We define the notion of a \emph{tame representation} of a group \( K \), and consider unitary representations of \( G \) that become tame after being restricted to \( K \). Such representations are called \emph{admissible} representations of the pair \( (G,K) \). We will also call them \( K \)-admissible representations of \( G \). Different choices of \( K \) may lead to different classes of admissible representations.

Different examples of \( (G,K) \)-pairs were studied in~\cites{Olsh89,Olsh90,Olsh91,Neretin}.

The pair \( (G,K) \) often allows one to construct a series \( \G(n) \) of finite semigroups in some natural way. This semigroup approach can be extremely useful as there are important connections between the representations of thus obtained semigroups and \( K \)-admissible representations of \( G \).

In this paper we consider pairs \( (\sym, K^{\alpha}) \), where \( \sym \) is the group of all finite permutations of the set \( \mathbb{N}:=\{1,2,\ldots\} \), and \( K^{\alpha} \) is the Young subgroup of \( \sym \) that respects a partition \( \alpha \) of \( \mathbb{N} \) into countable sets \( \mathbb{N} = \bigsqcup \alpha_i \).

The total classification of admissible representations of these types of pairs was firstly obtained by Nessonov in~\cite{Nessonov}. We will give another proof of these results making use of the semigroup approach.

\subsection{}

Let us denote by \( S(X) \) the group of finite permutations on the set \( X \). When \( X = \mathbb{N} \) and \( X = [n] := \{1,2,\ldots, n\} \) we will write \( S(\infty) \) and \( S(n) \) respectively. It will be convenient to us to realise \( S(X) \) as the group of all strictly monomial matrices indexed by the set \( X \).

\begin{definition}
	A \emph{partition} of a countable set \( X \) is a  collection of countable sets \( \alpha = \{\alpha_i\}_{i\in I} \) such that
	\[
	X = \bigsqcup_{i\in I} \alpha_i.
	\]
	The set of all partitions of \( X \) will be denoted by \( \mathcal{P}(X) \). When \( X \) equals \( \mathbb{N} \) we will omit \( \mathbb{N} \) from the notation, and sometimes, when we want to fix an index set \( I \), we will write \( \mathcal{P}^{I}(X) \).
\end{definition}
Each \( \alpha\in \mathcal{P} \) defines a group topology on \( \sym \) in the following way.

\begin{definition}\label{Definition:Topology}
	For any \( \alpha\in \mathcal{P}^{I}(X) \) let \( K^{\alpha} \) denote the Young subgroup of \( S(X) \) corresponding to \( \alpha \). It consists of all permutations that respect the partition \( \alpha \). It is clear that \( K^{\alpha} \) is a direct sum of symmetric groups on the sets \( \alpha_i \)
	\[
	K^{\alpha} = \sum_{i\in I} S(\alpha_i).
	\]
\end{definition}

For a finite subset \( Y\subset X \) let us denote by \( \bar{Y} \) the set \( \mathbb{N}\setminus Y\), by \( S_Y(X) \) the subgroup of \( \sym \) of all finite permutations of the set \( \bar{Y} \) (in other words, all permutations fixing the set \( Y \) pointwise), and set \( K_{Y}^{\alpha} = K^{\alpha}\cap S_Y(X) \). When \( Y = [n]\subset \mathbb{N} \) we will write \( \sym[n] \) and \( K_n^{\alpha} \) respectively.

Suppose now that \( X = \mathbb{N} \).

 Groups \( K_{n}^{\alpha} \) form a decreasing chain
\[
K^{\alpha} = K_{0}^{\alpha} \supset K_{1}^{\alpha} \supset \cdots
\]
and satisfy the following conditions:
\begin{enumerate}
	\item \( \bigcap_{n\in \mathbb{N}} K_{n}^{\alpha} = \{1\}\)\label{Definition:Topology1};
	\item \( K_{m}^{\alpha}\cap S(n) \) is finite for any \( m\leqslant n \)\label{Definition:Topology2};
	\item \( K_{n}^{\alpha} \) and \( S(n) \) commute for all \( n\in \mathbb{N} \).\label{Definition:Topology3}
\end{enumerate}
The conditions above guarantee the existence of the unique group topology on \( \sym \) for which \( \{K_{n}^{\alpha}\} \) form a neighbourhood basis of unity. We will call it \emph{\( \alpha \)-topology}.

\begin{remark}
	The completion of \( \sym \) in this topology consists of all bijections \( g \) of \( \mathbb{N} \) such that for any \( i\neq j\in I \) the number \(d_{ij}(g) =  \left| \{k\in \alpha_j \mid g(k)\in \alpha_i\} \right| \) is finite, and for any \( i\in I \) we have
	\[
	\sum_{j\neq i} d_{ij}(g) = \sum_{j\neq i} d_{ji}(g).
	\]
\end{remark}

\subsection{}

\begin{definition}
	Let \( \alpha\in \mathcal{P} \). A unitary representation \(T\) of  \(\sym\) is said to be \( \alpha \)-\emph{admissible} if it is continuous in \( \alpha \)-topology.
\end{definition}

For any unitary representation \( T \) of \( \sym \) we will denote by \( H(T) \) the Hilbert space of \( T \), and by \( H_{X}^{\alpha}(T) \) its subspace of all \( K_{X}^{\alpha} \)-invariant vectors. Once again, when \( X = [n] \), we will write just \( H^{\alpha}_n(T) \). The representation \( T \) is \( \alpha \)-admissible if and only if the space
\[
H_{\infty}^{\alpha}(T) := \bigcup_{n\in \mathbb{N}} H_{n}^{\alpha}(T)
\]
is dense in \( H(T) \).

When the partition \( \alpha \) is fixed we will usually omit the upper index \( \alpha \) from the notation.

\begin{remark}
	When the representation \( T \) is irreducible, it is \( \alpha \)-admissible if and only if \( H_n^{\alpha}\neq 0 \) for some \( n \).
\end{remark}

Note that if we change \( \alpha \) in a finite way, the category of admissible representations will not change. More formally, suppose we have two partitions \( \alpha \) and \( \beta \). We are interested in the relations between the corresponding categories \( \Adm({\alpha})\) and \( \Adm({\beta}) \) of equivalence classes of admissible representations.

\begin{notation}
	For any finite set \( Y\subset \mathbb{N} \) let us denote by \( \alpha^{(Y)}\) the partition of \( Y^{c} \) obtained from \( \alpha \) by removing the set \( Y \). In other words, \( \alpha^{(Y)}_{i} = \alpha_{i}\cap \bar{Y} \). When \( Y = [n] \) we will write \( \alpha^{(n)} \).
\end{notation}
\begin{definition}
	Let \( \alpha\in \mathcal{P}^{I} \), \(\beta \in \mathcal{P}^{J} \) be two partitions of \( \mathbb{N} \).
	\begin{enumerate}
		\item  We say that two partitions \( \alpha \) and \( \beta \) are \emph{equivalent}, \( \alpha\sim \beta \), if \( \alpha^{(Y)} = \beta^{(Y)} \) for some finite set \( Y \).

		Note that \( \alpha\sim \beta \) if and only if \( K^{\alpha}_n = K^{\beta}_n \) for all \( n \) large enough. It follows that they yield the same categories \( \Adm({\alpha}) = \Adm({\beta})\).

		We will denote by \( [\alpha] \) the equivalence class of \( \alpha \).

		\item We say that \( \alpha \) is \emph{finer} than \( \beta \) (and that \( \beta \) is \emph{coarser} than \( \alpha \)) and write \( \alpha\succcurlyeq \beta \) if for any \( j\in J \) there exists \( I_j\subset I \) such that
		\[
		\beta_j = \bigsqcup_{i\in I_j} \alpha_i.
		\]
		Note that if \( \alpha \succcurlyeq \beta \), then \( K^{\alpha}\subset K^{\beta} \).

		We will write \( \alpha = \beta \) when \( \alpha \succcurlyeq \beta \) and \( \beta \succcurlyeq \alpha \), i.e., when \( \beta \) can be obtained from \( \alpha \) by re-indexing its parts.

		\item We say that class \( [\alpha] \) is \emph{finer} than class \( [\beta] \), \( [\alpha]\succcurlyeq [\beta] \), if \( \alpha^{(Y)}\succcurlyeq \beta^{(Y)} \) for some finite set~\( Y \). This definition does not depend on the choice of representatives \( \alpha \) and \( \beta \).

		Note that \( [\alpha]\succcurlyeq [\beta] \) if and only if \( K^{\alpha}_n \subset K^{\beta}_n \) for all \( n \) large enough, and if \( [\alpha]\succcurlyeq [\beta] \) and \( [\beta] \succcurlyeq[\alpha] \), then \( \alpha\sim \beta \).

		It follows that  \( \Adm({\alpha})\supset \Adm({\beta}) \) when \( [\alpha ]\succcurlyeq [\beta ]\).
	\end{enumerate}
\end{definition}

The structure of this paper is as follows. In Section~\ref{Section:Semigroups} we introduce semigroups of double cosets \( \G^{\alpha}(X) \) and describe the connection between their representations and \( \alpha \)-admissible representations of \( \sym \). In Section~\ref{Section:Spherical} we describe some particular subclass of \( \alpha \)-admissible representations, namely, spherical representations. The proof of the classification of \( \alpha \)-spherical representations using the semigroup approach was done by Neretin in~\cite{Neretin} and we will not repeat it. In Section~\ref{Section:Admissible} we use the classification of \( \alpha \)-spherical representations to give the semigroup-theoretical proof of the classification result of irreducible \( \alpha \)-admissible representations, and in Section~\ref{Section:Equivalence} give the necessary and sufficient conditions on unitary equivalence of two irreducible \( \alpha \)-admissible representations. Section~\ref{Section:Intersection} is an application of the classification result: we describe irreducible representations that are admissible with respect to two partitions simultaneously.

\begin{acknowledgements}
	This work was supported by the Russian Science Foundation under project~23--11--00150, \url{https://rscf.ru/en/project/23-11-00150/}. The author is grateful to G.~Olshanskii for stating the problem, useful discussions and careful reading of the paper.
\end{acknowledgements}

\section{Semigroups of double cosets}\label{Section:Semigroups}

In this section we describe the set \( \G^{\alpha}(X) \) of double cosets \( K_{X}^{\alpha}\backslash \sym / K_{X}^{\alpha} \) and endow it with a semigroup structure.

\subsection{}

Firstly, we agree that in this paper when we say ``semigroup'' we mean ``involutive semigroup with unity'', i.e., we require a semigroup \(\G\) to possess a unity \(1\) and an involutive anti-homomorphism \(\g\mapsto \g^{\ast}\).
\begin{definition}
	Let \(X\) be a set.
	\begin{enumerate}[(i)]
		\item  A \emph{partial bijection} of \(X\) is a bijection \(A\colon \dom(A) \rightarrow \range(A)\) between two (possible empty) subsets \(\dom(A), \range({A}) \subset X\). The set of all partial bijections of \(X\) will be denoted by~\(\PB(X)\).
		\item For  \(A_1, A_2 \in \PB(X)\) we define their product \(A_1 A_2\) in the natural way: \(A_1 A_2\) is defined on \(x\) whenever both \(A_2(x)\) and \(A_1(A_2(x))\) are defined,
		and for any \(x\in \dom(A_1A_2)\)
		\[
		A_1 A_2 (x) = A_1(A_2 (x)).
		\]
		\item Given \(A\in \PB(X)\) we define \(A^{\ast}\in \PB(X)\) such that \(\dom(A^{\ast}) = \range(A)\), \(\range(A^{\ast}) = \dom(A)\), and for any \(x\in \dom(A^{\ast})\)
		\[
		A^{\ast}(x) = A^{-1}(x).
		\]
	\end{enumerate}
	Under these operations \(\PB(X)\) becomes a semigroup with an involution \(A^{\ast}\) and a unity \(\id\colon X\rightarrow X\).
\end{definition}

\begin{notation}
	If \( X' \) is a subset of \( X \), we will denote by \( 1_{X'} \) an idempotent defined by
	\[
	1_{X'}(x) = \begin{cases}
		x, & x\in X';\\
		\text{not defined}, & x\notin X'.
	\end{cases}
	\]
\end{notation}

\begin{remark}
	It is convenient to realise semigroup \( \PB(X) \) as the semigroup of all not strictly monomial matrices, i.e., matrices \( A = (A_{xy})_{x,y\in X} \) defined by
	\[
	A_{xy} = \begin{cases}
		1& \text{when \( A(y) = x \)}\\
		0& \text{otherwise}.
	\end{cases}
	\]
	The product of partial bijections is the matrix product and \( A^{\ast} = A^t \) is the transpose of \( A \).
\end{remark}

Let us now describe the set of double cosets \( K_{X}^{\alpha}\backslash \sym /K_{X}^{\alpha} \) for some finite subset \( X \)  of \( \mathbb{N} \).

\begin{proposition}\label{Proposition:DoubleCosets}
	For any \( g\in \sym \) the double coset \( K_{X}^{\alpha}gK_{X}^{\alpha} \) is uniquely determined by the data \( \g = (A^{X}(g), B_i^{\alpha, X}(g), C_j^{\alpha, X}(g), d^{\alpha, X}_{ij}(g) \mid i,j\in I) \) defined as follows:
	\begin{enumerate}
		\item\label{DCData:X} \( A^{X}(g)\) is a partial bijection on the set \( X \) obtained from \( g \) by taking the submatrix of \( g \) corresponding to the set \( X \). More formally,
		\begin{gather*}
		 \dom(A^{X}(g)) = \{k\in X\mid g(k)\in X\},\quad \range(A^{X}(g)) = g(\dom(A^{X}(g))),\\ A^{X}(g)(k) = g(k);
		 \end{gather*}
		\item\label{DCData:A} \( B_i^{\alpha, X}(g)\) is a subset of \( X \) defined by \[ B_i^{\alpha, X}(g) = \{k\in X\mid g(k)\in \alpha_i^{(X)}\},\quad i\in I;\]
		\item\label{DCData:B} \( C_j^{\alpha, X}(g)\) is a subset of \( X \) defined by \[ C_{j}^{\alpha, X}(g) = \{k\in X\mid g^{-1}(k)\in \alpha_j^{(X)}\},\quad j\in I;\]
		\item\label{DCData:c} For \( i\neq j \) the number \(d_{ij}^{\alpha, X}(g)\) is the size of the finite set \[ D_{ij}^{\alpha, X}(g) = \{k\in \alpha_j^{(X)}\mid g(k)\in \alpha_i^{(X)}\},\quad i\neq j \] and \( d_{ij}^{\alpha, X}(g) = \infty \) when \( i = j \).
	\end{enumerate}
\end{proposition}
When there is no confusion we will omit \( \alpha \) from the notation. In the proof of the proposition we assume that \( \alpha \) is fixed.

\begin{proof}
	It is clear that this data does not change after multiplication by elements of \( K_{X} \) on either side.

	Now suppose that \( g, h\in \sym \) have the same data. We may assume that \( g,h \) lie in \( S(M) \) for some \( M \) big enough. Let us write \( D_{ii}(g) \) for the set \mbox{\( \{k\in \alpha_i^{(X)}\mid g(k)\in \alpha_i^{(X)}\} \cap [M]\)}. Note that \( \left|{D_{ii}(g)}\right| = \left|{D_{ii}(h)} \right| \) for all \( i\in I \).

	We want to find \( k_1, k_2 \in K_{X}\) such that \( k_1h = gk_2 \). Let us fix for any \( i, j \) some bijection \(\varphi_{ij}\) from the set \( D_{ij}(h) \) to the set \( D_{ij}(g) \). Then we set
	\begin{gather*}
		k_1(x) = \begin{cases}
			(gh^{-1})(x)& \text{for \( x\in h(B_i(h)) \)};\\
			(g\varphi_{ij} h^{-1})(x)& \text{for \( x\in h(D_{ij}(h)) \)};\\
			x& \text{otherwise}.
		\end{cases}\quad
		k_2(k) = \begin{cases}
			(g^{-1}h)(x)& \text{for \( x\in h^{-1}(C_i(h)) \)};\\
			\varphi_{ij}(x)& \text{for \( x\in D_{ij}(h) \)};\\
			x& \text{otherwise}.
		\end{cases}
	\end{gather*}
	It is easy to check that \( k_1, k_2 \) satisfy the desired properties.
\end{proof}

There is a more convenient way to write this data. Put \( \g = (A, (B_i), (C_i), (d_{ij})) \) as above. We will write \( A \) as a finite monomial matrix \( A = (A_{kl})_{k,l\in X} \), the set \( B_i \) as a row vector indexed by \( X \), \( B_i = (B_{ik})_{k\in X} \)
\[
B_{ik} = \begin{cases}
	1& \text{if \(k\in B_i  \)}\\
	0& \text{if \(k\notin B_i  \)}
\end{cases}
\]
 and the set \( C_j \) as a column vector \( C_j = (C_{jk})_{k\in X} \) in the similar way.

Now we may write \( \g\) as a following (infinite when \( I \) is infinite) matrix with coefficients in \( \mathbb{Z}_{\geqslant 0}\cup\{\infty\} \):
\begin{gather}
	\g = \begin{bmatrix}\begin{array}{c|c c c}
			A & C_{i_1} & C_{i_2} & \cdots\\
			\hline
			B_{i_1} & \infty & d_{i_1i_2} &\dots\\
			B_{i_2} & d_{i_2i_1} & \infty &  \dots\\
			\vdots & \vdots & \vdots &  \infty
	\end{array}\end{bmatrix}.
\end{gather}
More formally, \( \g \) becomes a matrix indexed by the set \( X\sqcup I \) defines as follows:
\begin{gather}
	\g_{ab} = \begin{cases}
		A_{ab}& \text{if \( a,b\in X \)};\\
		C_{ba} & \text{if \(  a\in X, b\in I \)};\\
		B_{ab} & \text{if \( a\in I, b\in X \)};\\
		d_{ab} & \text{if \( a,  b\in I \)}.
	\end{cases}
\end{gather}
For the sake of brevity, we will write
\[
\g = \begin{bmatrix}
	A & (C_{j})\\
	(B_{i}) & (d_{ij})\\
\end{bmatrix}.
\]

\begin{definition}
	For any finite set \( X\) and any decomposition \( \alpha\in \mathcal{P}^{I} \) let \( \G^{\alpha}(X) \) be the set of all matrices \( \g \) indexed by the set \( X\sqcup I \) with coefficients in \( \mathbb{Z}_{\geqslant 0}\cup\{\infty\} \) such that
	\begin{enumerate}
		\item for any \( k\in X \) the corresponding row \( (\g_{\ast k}) \) (and column \( (\g_{k\ast}) \)) has exactly one entry equal to \( 1 \) and all the rest equal to \( 0 \);
		\item for any \( i\in I \) the entry \( \g_{ii} \) equals \( \infty \);
		\item for any \( i\neq j\in I \) the entry \( \g_{ij} \) lies in \( \mathbb{Z}_{\geqslant0} \) and only finitely many of them are not zero;
		\item for any \( i\in I \) holds
		\[
		\sum_{k\in X}\g_{ik} + \sum_{I\ni j\neq i} \g_{ij} = \sum_{I\ni j\neq i} \g_{ji} + \sum_{k\in X} \g_{ki}.
		\]
	\end{enumerate}
\end{definition}

It is clear that for any such \( \g \) we can find an element \( g\in \sym \) such that \( K_{X}gK_{X} \) corresponds to \( \g \). So we have the following

\begin{proposition}
	For any partition \( \alpha \) of the set \( \mathbb{N} \) the map \( \theta^{\alpha,X}\colon g\mapsto \g \) described in the statement of Proposition~\ref{Proposition:DoubleCosets} gives a bijection between the set of double cosets \( K^{\alpha}_{X}\backslash \sym / K^{\alpha}_{X} \) and \( \G^{\alpha}(X) \).
\end{proposition}

Now we want to define on \( \G^{\alpha}(X) \) a structure of a semigroup.

Let us once again omit \( \alpha \) from notation.
\begin{definition}
	For any permutation \( g \) in \( S(X) \) its \emph{support} is the set of all points in \( X \) that are not fixed under the action of \( g \).

	Let \( \mathfrak{f}, \mathfrak{h} \) be from \( \G(X) \), and \( f, h \) be their corresponding representatives in \( \sym \). We say that representatives \( f \) and \( h \) are \emph{in general position} if the size of the intersection of their supports is as small as possible. We define the multiplication in \( \G(X) \) as
	\[
	\mathfrak{f}\cdot \mathfrak{h} = \theta^{X}(fh).
	\]
\end{definition}

\begin{proposition}
	The multiplication is correctly defined (i.e., doesn't depend on the choice of representatives in general position) and is associative.
\end{proposition}

\begin{proof}
	Since \( f \) and \( h \) are in general position, we have
	\begin{enumerate}[(i)]
		\item \( h(B_i(h)) \cap \supp(f)  = f^{-1}(C_i(f))\cap\supp(h) = \emptyset\) for any \( i\in I \);\label{SemiMult:1}
		\item \( D_{ij}(f) \cap \supp(h) = D_{ij}(h) \cap \supp(f) = \emptyset\) for any \( i\neq j\in I \);\label{SemiMult:2}
		\item \( f(D_{ij}(f)) \cap \supp(h) = h(D_{ij}(h)) \cap \supp(f) = \emptyset\) for any \( i\neq j\in I \)\label{SemiMult:3}.
	\end{enumerate}
	Let now \( g = fh \). Then conditions above give us the following formulas (all multiplications are multiplications as matrices). We give below informal explanations, but the formulas can be checked directly.
	\begin{enumerate}[1.]
		\item \( A(g) = A(f)A(h) \). To see this, observe that condition~\eqref{SemiMult:1} ensures that any point \( k\in X \) that leaves the set \( X \) under the action of \( h \) cannot return back under \( f \). So we have \( \dom(A(g)) = \dom(A(h))\cap h^{-1}(\dom(A(f))) \) and \( A(g) = A(f)\cdot A(h) \) as partial bijections.
		\item \( B_i(g) = B_i(h) + B_i(f)A(h) \). An element \( k\in X \) moves to the set \( \alpha_i^{X} \) under the action of \( g \) in exactly two cases: either it is moved there by \( h \) (in which case \( f \) does not move it again), or it stays in \( X \) under \( h \) and is then moved to \( \alpha_i^{(X)} \) by \( f \).
		\item \( C_j(g) = C_j(f) + A(f)C_j(h) \). To see this, observe that \( C_{j}(g) = B_j(g^{-1})^{t} \).
		\item \( d_{ij}(g) = d_{ij}(h) + d_{ij}(f) + B_{i}(f)C_{j}(h) \). An element \( k\in \alpha_j^{(X)} \) moves to \( \alpha_i^{(X)} \) in three cases: it was moved there by \( h \) (and is left there by \( f \)); it was left in \( \alpha_j^{X} \) by \( h \), and then moved by \( f \); or it was moved by \( h \) to \( X \) and then moved to \( \alpha_i^{(X)} \) by \( f \).
	\end{enumerate}
	So we have the following formula for the multiplication:
	\begin{multline}\label{Equation:SemiGrMult}
		\begin{bmatrix}
			A(f) & (C_j(f))\\
			(B_i(f)) & (d_{ij}(f))\\
		\end{bmatrix} \cdot
		\begin{bmatrix}
			A(h) & (C_j(h))\\
			(B_i(h)) & (d_{ij}(h))\\
		\end{bmatrix} = \\
		=\begin{bmatrix}
			A(f)A(h) & (C_j(f) + A(f)C_j(h))\\
			(B_i(h) + B_i(f)A(h)) & (d_{ij}(f) + d_{ij}(h)+B_i(f)C_j(h))\\
		\end{bmatrix}.
	\end{multline}
	From this formula the correctness follows automatically and the associativity can be checked directly.
\end{proof}

When \( X \) is the empty set, we will write \( \G^{\alpha} \) instead of \( \G^{\alpha}(\emptyset) \).

\begin{example}
	The semigroup \( \G^{\alpha} \) is realised as the semigroup of matrices \( \mathfrak{g} = [d_{ij}]_{i\neq j\in I} \) with \( d_{ij}\in \mathbb{Z}_{\geqslant 0} \) satisfying the condition
	\[
	\sum_{j\neq i} d_{ij} = \sum_{j\neq i} d_{ji}
	\]
	and with semigroup operation given by matrix addition. It is therefore a commutative semigroup.
\end{example}

We define an involution on \( \G(X) \) by matrix transposition. One can see that \( (\theta_{X}(g))^{\ast} = \theta_{X}(g^{-1}) \).

There is another equivalent way to define multiplication on \( \G(X) \).

\begin{notation}
	Let \( Y \) be some finite subset of \( \mathbb{N} \), and let \( Y = X\sqcup X' \). We fix some element \( w^{\alpha}_{Y,X} \) of \( \sym \) such that
	\begin{enumerate}[(i)]
		\item \( w^{\alpha}_{Y,X} \) lies in \( K_{X} \);
		\item \( w^{\alpha}_{Y,X} (X') \subset \bar{Y} \);
	\end{enumerate}
\end{notation}

\begin{proposition}
	The multiplication in \( \G^{\alpha}(X) \) can be defined in the following way: let \( f,h \) be arbitrary elements of \( \sym \) and \( Y \) be such that \( X\subset Y \) and \( f,h\in S(Y) \). Then
	\[
	\theta^{\alpha, X}(f)\cdot\theta^{\alpha, X}(h) = \theta^{\alpha, X}(fw^{\alpha}_{Y,X}h).
	\]
	The result does not depend on the choice of \( Y \) and \( w^{\alpha}_{Y,X} \).
\end{proposition}
The proof of this proposition is a straightforward, but rather unpleasant check.

We will denote the image of \( w^{\alpha}_{Y,X} \) under \( \theta^{\alpha, Y} \) by \( \epsilon^{\alpha}_{Y,X} \). One can easily see that
\[
\epsilon^{\alpha}_{Y,X} = \begin{bmatrix}
	1_{X} & \big( X'\cap \alpha_i \big)\\
	\big( X'\cap \alpha_i \big) & \big(0\big)
\end{bmatrix}.
\]
\subsection{}
Let \(H\) be a complex Hilbert space of finite or countable dimension. A \emph{contraction} on \(H\) is an operator with norm \(\le1\). Let \(C(H)\) be the set of contractions. It is a semigroup with involution (the conventional conjugation of bounded operators). If \(H\) has countable dimension, we endow \(C(H)\)  with the weak operator topology.

By a \emph{representation} of a semigroup \(\G\) on \(H\) we mean a homomorphism \(\T\colon \G\to C(H)\) which preserves the unity and is compatible with the involution, that is,  \(\T(\g^\ast)=(\T(\g))^\ast\) for all \(\g\in\G\).

Let \( T \) be some unitary representation of \( \sym \). For any partition \( \alpha\in \mathcal{P} \) and a finite set  \( X \), let \( P^{\alpha}_{X} \) denote the orthogonal projection to the space \( H^{\alpha}_{X}(T) \) of \( K_X^{\alpha} \)-invariant vectors. Then for any \( g\in \sym \) the operator
\[
P^{\alpha}_{X}\rest{T(g)}{H^{\alpha}_{X}(T)}\colon H^{\alpha}_{X}(T)\rightarrow H^{\alpha}_{X}(T)
\]
depends only on the double coset \( K_{X}^{\alpha}gK_{X}^{\alpha} \). Thus we have a correctly defined map
\[
\T^{\alpha}_{X}\colon \G^{\alpha}(X) \rightarrow \End(H^{\alpha}_{X}(T)), \quad \T^{\alpha}_{X}\colon \theta^{\alpha, X}(g)\mapsto P^{\alpha}_{X}\rest{T(g)}{H^{\alpha}_{X}(T)}.
\]
Let us again omit \( \alpha \) from notations for the rest of this section.
\begin{proposition}
	For any unitary representation \( T \) of \( \sym \) and any finite set \( X \) the corresponding map \( \T_{X} \) is a representation of a semigroup \( \G(X) \). It is irreducible when \( T \) is irreducible, and if \( T \) is also admissible, then \( T \) is uniquely defined by \( \T_{X} \) for any \( X \) s.t. \( H_{X}(T) \) is not \( 0 \).
\end{proposition}

The proof of this proposition is similar to the proofs of similar results for other pairs \( G, K \), for example, see~\cite[Theorem~2.5]{Olsh89}. The important step in the proof is the following
\begin{lemma}
	For any finite sets \( X \subset Y\subset \mathbb{N} \) holds
	\[
	P_{Y}T(w_{Y,X})P_{Y} = P_{X}.
	\]
	In particular, the orthogonal projection
	\[
	P_{Y, X}\colon H_{Y}(T)\rightarrow H_X(T)
	\]
	is given by \( \T_{Y}(\epsilon_{Y,X}) \).
\end{lemma}

\section{Spherical representations}\label{Section:Spherical}

\subsection{}
\begin{definition}
	Let \( K \) be a subgroup of a group \( G \). We say that \( (G,K) \) is a \emph{Gelfand pair} if for any unitary representation \( T \) of \( G \)  and any \( g,h\in G \) the operators \( PT(g)P \) and \( PT(h)P \) commute. Here \( P \) is the orthogonal projection onto the space \( H(T)^{K} \) of \( K \)-invariant vectors.
\end{definition}

\begin{remark}
	Note that if \( (G,K) \) is a  Gelfand pair, it follows that for any irreducible unitary representation \( T \) the space \( H(T)^{K} \) is at most one-dimensional.
\end{remark}

\begin{definition}
	A unitary representation of \( G \) is a \emph{spherical representation} of a Gelfand pair \( (G,K) \) if it possesses a cyclic \( K \)-invariant vector \( \xi \) (it means that the orbit of \( \xi \) under the action of \( G \) is total in \( H(T) \)). We will always assume that \( \|\xi\| = 1 \).

	Such a vector will be called a \emph{spherical vector}, and the corresponding matrix element \( \varphi(g) := \langle T(g)\xi; \xi \rangle \) will be called a \emph{spherical function} of \( T \). It is well-known that any irreducible spherical representation if uniquely determined by its spherical function (see, for example,~\cite[Chapter~8]{BorOlsh})
\end{definition}

\begin{proposition}[cf. {\cite[Proposition~3.6]{Nessonov}}]
		For any partition \( \alpha\in \mathcal{P} \) the pair \( (\sym; K^{\alpha}) \) is a Gelfand pair.

		Let \( T \) be a spherical representation of the pair \( (\sym, K^{\alpha}) \). Then the corresponding representation \( \T \) of the semigroup \( \G^{\alpha} \) is one-dimensional and the spherical function \( \varphi(g) \) of \( T \) is given by the character of the representation \( \T \) of the semigroup \( \G^{\alpha} \).

\end{proposition}

\begin{proof}
	Recall that \( \G^{\alpha} \) is a commutative semigroup. So for any irreducible representation \( T \) of \( \sym \) the operators \( PT(g)P = \T(\theta^{\alpha}(g)) \) commute.

	The second part follows immediately from the fact that \( \xi \) is \( K \)-invariant.
\end{proof}

The total classification of spherical representations of the pair \( (\sym, K^{\alpha}) \) was first obtained by Nessonov in~\cite{Nessonov}. In~\cite{Neretin} Neretin gives another proof, using the semigroup approach. Strictly speaking, Neretin proves this result only in the case of partitions into finitely many parts (i.e., \( |I| = m<\infty \)), but in fact, all of his arguments work for the infinite case.

For the sake of brevity we will call unitary representation \( T \) of \( \sym \) that is spherical to a pair \( (\sym, K^{\alpha}) \) an \( \alpha \)-spherical representation. Now we will show how to construct irreducible \( \alpha \)-spherical representations of \( \sym \).

\subsection{}
Let us recall the definition of the countable tensor product of Hilbert spaces.

Suppose we have a collection of Hilbert spaces \( (V)_{k\in \mathbb{N}} \) and let us fix some unit vectors \( \xi_k\in V_k \).

For any \( n\leqslant m \) we define an inclusion
\[
\bigotimes_{i = k}^{n} V_k \rightarrow \bigotimes_{k = 1}^{m} V_k
\]
 by
\[
v_1\otimes \cdots \otimes v_n \mapsto v_1\otimes \cdots \otimes v_n\otimes \xi_{n+1} \cdots \otimes \xi_m.
\]
The inductive limit of this direct system is an inner product space and we will denote its completion by \( \bigotimes_{k\in \mathbb{N}} (V_k, \xi_k) \).

When all \( V_k \) are copies of the same Hilbert space \( V \) we will write \( V^{\otimes \infty}(\xi) \), where \( \xi \) denotes the vector \( \xi = \xi_1\otimes \xi_2\otimes \ldots \)

Let us choose in each \( V_k \) some orthonormal basis \( \xi_k = e^{(k)}_0, e^{(k)}_1, e^{(k)}_2, \ldots \) Then \( \bigotimes_{k\in \mathbb{N}} (V_k, \xi_k) \) has an orthonormal basis \( e_f \)
\[
e_f = e^{(1)}_{f(1)}\otimes e^{(2)}_{f(2)}\otimes \cdots,
\]
with almost all values \( f(k) \) equal to \( 0 \).

\subsection{}

Now we will construct some representations.

Fix a partition \( \alpha \in \mathcal{P}^{I} \).

Let \( V \) be a Hilbert space, and \( (v_{i})_{i\in I} \) be some unit vectors generating \( V \) and such that \( v_i \) and \( v_j \) are not collinear for \( i\neq  j \). We take a countable number of copies of \( V \) and choose \( \xi_k = v_i \)  when \( k \) lies in \( \alpha_{i} \).

The infinite symmetric group \( S(\infty) \) acts on the space \( V^{\otimes \infty}(\xi) \) by permuting the terms.

 Let \( H \) be the closure of the cyclic span of \( \xi \) under this action, i.e., \( H = \overline{\spanset\{T(g)\xi \mid g\in \sym\}} \), and let \( S \) be the restriction of the representation above to \( H \). The following theorem holds.

 \begin{theorem}\label{Theorem:Spherical}
 	\begin{enumerate}[1.]
 		\item This representation is irreducible. It is \( \beta \)-spherical for any \( \beta\succcurlyeq \alpha \).

 		\item (cf.~\cite[Theorem~3.4]{Nessonov}) Its spherical function is
 		\[
 		\varphi(g) = \prod_{i\neq j \in I} \langle v_j, v_i\rangle^{d_{ij}^{\alpha}(g)}
 		\]
 		(here we assume that \( 0^0 = 1 \)).

 		It follows that \( S \) is uniquely determined by a partition \( \alpha \) and an \mbox{\( I\times I \)} Gram matrix \( G = ( \langle v_j, v_i \rangle )_{i,j\in I} \). We will denote this representation by \( S(\alpha, G) \).

 		\item (cf.~\cite[Proposition~4.10]{Nessonov}) Two representations \( S(\alpha, G) \) and \( S(\alpha, G') \) are equivalent if and only if there is a complex diagonal matrix \( D = (d_i)_{i\in I} \) with \( |{d_i}| = 1 \), such that \( G' = DGD^{-1} \). In other words, if and only if the corresponding systems of vectors \( (v_i)\), \((v_i') \) can be obtained from each other by composition of an isometry and a diagonal map \( v_i\mapsto d_{i}v_i \), \( |d_i| = 1 \).

 		This gives us an equivalence relation on Gram matrices. We will denote by \( \mathcal{G}^{I} \) the set of some representatives of these equivalence classes.

 		\item (cf.~\cite[Theorem~3.7]{Nessonov},~\cite[Theorem~8.1]{Neretin}) All \( \beta \)-spherical representations can be obtained in this way. I.e., if \( S \) is some \( \beta \)-spherical representation, then there is a partition \( \alpha\in \mathcal{P}^{I} \) such that \( \alpha\preccurlyeq \beta \) and a Gram matrix \( G\in \mathcal{G}^{I} \) such that \( S \) is equivalent to \( S(\alpha, G) \).
 	\end{enumerate}

 	\begin{proof}
 		We will not prove the last statement.
 		\begin{enumerate}[1.]
 			\item 	The space of \( K_n^{\alpha} \)-invariant vectors in \( V^{\otimes \infty}(\xi) \) coincides with \( V^{\otimes n} \). It follows that \( \xi \) is the only (up to multiplication by scalar) \( K^{\alpha} \)-invariant vector in \( V^{\otimes \infty}(\xi) \), hence \( S \) is irreducible. Indeed, if \( H(S) = H\oplus H' \) and \( P \) is the orthogonal projection to \( H \), then \( P(\xi) \) is \( K^{\alpha}\)-invariant and, therefore, either \( H = H(S) \) or \( H' = H(S) \). It is also clear that \( \xi \) is \( K_{\beta} \)-invariant for any partition \( \beta \) finer then \( \alpha \).

 			\item The spherical function is
 			\begin{gather*}
 				\varphi(g) = \langle T(g)\xi, \xi \rangle = \prod_{k} \langle \xi_{g^{-1}(k)}, \xi_{k} \rangle = \prod_{i\neq j} \langle v_j, v_i\rangle^{d_{ij}^{\alpha}(g)}.
 			\end{gather*}

 			\item In~\cite[Lemma~8.2]{Neretin} Neretin shows that the semigroup \( \G^{\alpha} \) is generated by cycles \( \mathfrak{s}[i_1, \ldots, i_p] \) that are defined as follows: if we write \( \G^{\alpha} \) as a semigroup of matrices, then any element from \( \G^{\alpha} \) can be written as a linear combination (with coefficients in positive integers) of elementary matrices \( E_{ij} \) for \( i\neq j \) (with \( \infty \) on the diagonal). Then for any \( p\geqslant 2 \) and any pairwise distinct \( i_1, \ldots, i_p \in I\) we define cycles as elements
 			\[
 			\mathfrak{s}[i_1\ldots, i_p] = E_{i_1i_2} + E_{i_2i_3} + \cdots + E_{i_pi_1}.
 			\]
 			Any element \( \mathfrak{g}\in \G^{\alpha} \) can be represented as a finite product of cycles.

 			Now, two irreducible spherical representations \( S \) and \( S' \) are equivalent if and only if their corresponding spherical functions \( \varphi \) and \( \varphi' \) are the same. Or, equivalently, if the characters \( \chi, \chi' \) of the semigroup \( \G^{\alpha} \) are the same. Now, it is enough to compute the value of \( \chi \) and \( \chi' \) on cycles \( \mathfrak{s}[i_1, \ldots, i_p] \). We have the following formula
 			\[
 			\chi(\mathfrak{s}[i_1, \ldots, i_p]) = g_{i_1i_2}\cdots g_{i_pi_1}.
 			\]
 			So we are left to prove the following.
 			\begin{lemma}
 				Let \( G, G' \) be two Gram matrices indexed by the set \( I \) such that \( g_{ii} = g'_{ii} = 1 \). Then \( G' = DGD^{-1} \) for some diagonal complex matrix \( D \), \( |d_i| = 1 \), if and only if for any \( p\geqslant 2 \) and any pairwise distinct \( i_1, \ldots, i_p\in I \) the following  holds:
 				\begin{gather}
 				g_{i_1i_2}\cdots g_{i_pi_1} = g'_{i_1i_2}\cdots g'_{i_pi_1} \label{Equation:Cycles}.
 				\end{gather}
 			\end{lemma}
 			\begin{proof}
 				Note that the condition \( G' = DGD^{-1} \) means that \( g'_{ij} = g_{ij}\frac{d_i}{d_j} \) for any \( i,j\in I \).

 				The ``only if'' part is trivial. Now suppose that~\eqref{Equation:Cycles} is satisfied. In particular, \( |g_{ij}|^2 = |g'_{ij}|^2 \)  and \( g_{ij} = 0 \) if and only if \( g'_{ij} = 0 \).

 				Let \( \Gamma \) be undirected graph with vertices indexed by \( I \) and such that two vertices \( i, j \) are connected if and only if \( g_{ij}\neq 0 \). Let \( C \) denote the set of all connected components \( c \) and let us  choose in every component \( c \) a spanning tree \( \Gamma_c \) with a fixed root \( i_c \). For any \( i\in c \) we set the value \( d_i \) by induction on the length of the path from \( i_c \) to \( i \). We set \( d_{i_c} = 1 \) and for any path \( i_c,\ldots, j, i \) we set \( d_i = d_j\frac{g'_{ij}}{g_{ij}} \).

				We defined all \( d_i \) for \( i\in c \), and if \( i \) and \( j \) are connected by an edge in the tree we have the required property. Now suppose that \( i \) and \( j \) are connected in \( \Gamma \), but not in \( \Gamma_c \). Then if we add this edge to \( \Gamma_c \) we will obtain a cycle \( j = i_1, i_2, \ldots, i_p = i \). Applying~\eqref{Equation:Cycles} we obtain
				\[
				\frac{g'_{ij}}{g_{ij}} = \frac{g_{i_1i_2}\cdots g_{i_{p-1}i_p}}{g'_{i_1i_2}\cdots g'_{i_{p-1}i_p}} = \frac{d_{i_2}\cdots d_{i_{p}}}{d_{i_1}\cdots d_{i_{p-1}}} = \frac{d_i}{d_j}.
				\]
 			\end{proof}
 		\end{enumerate}
 	\end{proof}

 	\begin{remark}
 		In the construction of \( S(\alpha, G) \) we demanded that \( v_i, v_j \) must not be collinear for \( i\neq j \). Suppose that \( \tilde{S} \) is a representation constructed from partition \( \beta \) and vectors \( v_i \) in the manner described above, but without the non-collinearity condition on \( v_i \). Then \( \tilde{S} \) is equivalent to \( S(\alpha, G) \), where \( \alpha\preccurlyeq \beta \) is obtained from \( \beta \) by uniting all \( \alpha_i, \alpha_j \) such that \( \mathbb{C}v_i = \mathbb{C}v_j \), and \( G \) is the corresponding Gram matrix.

 		We could describe \( \beta \)-spherical representations either as \( S(\alpha, G) \) for \( \alpha\preccurlyeq\beta \) with non-collinear condition or as representations \( S(\tilde{G}) \) depending only on Gram matrix \( \tilde{G} \) without the condition. For several reasons, the first description is more convenient to us.
 	\end{remark}
 	\begin{example}\label{Example:Spherical}
 		When \( \beta \) is a partition of \( \mathbb{N} \) onto two sets \( \beta_1, \beta_2 \) any \( \beta \)-spherical representation is uniquely determined by a real number \( a\in [0,1] \). Indeed, this representation is defined by the value \( g_{12} = \langle v_2, v_1 \rangle \). But since we may multiply \( v_2 \) by any scalar \( |d| = 1 \), we may assume that \( g_{12} \) is real and positive. The value \( g_{12} = 1 \) corresponds to the trivial representation, \( g_{12} = 0 \) to the  representation induced from the trivial representation of \( K^{\beta} \).
 	\end{example}
 \end{theorem}

\section{Admissible representations}\label{Section:Admissible}
In this section we construct some representations \( T(\alpha, n, \lambda, G) \) and later prove that any \( \beta \)-admissible representation can be obtained in this way.

Let us fix some finite number \( n \). The irreducible representation of \( S(n) \) corresponding to a Young diagram \( \lambda \) will be denoted by \( R^{\lambda} \).

The pair \( (\sym[n], K^{\beta}_n) \) is a Gelfand pair and their spherical representations are determined by some partition \( \alpha^{(n)}\in \mathcal{P}^{I}(\bar{[n]}) \) such that \( \alpha^{(n)}\succcurlyeq \beta^{(n)} \) and \( G\in \mathcal{G}^{I} \). We will denote this representation by \( S(\alpha^{(n)}, G) \).

 We can construct the induced representation
\[
T(\alpha, n, \lambda, G) = \Ind^{\sym}_{S(n)\sym[n]} R^{\lambda} \otimes S(\alpha^{(n)}, G).
\]
\begin{theorem}\label{Theorem:IrrAndAdm}
	The representation \( T(\alpha, n, \lambda, G) \) is irreducible and \( \alpha \)-admissible.
\end{theorem}

\begin{proof}
	Denote the subgroup \( S(n)\sym[n] \) by \( L \) and the representations \( S(\alpha^{(n)}, G) \) and \( T(\alpha, n, \lambda, G) \) by \( S \) and \( T \) respectively.

	The space \( H(T) \) consists of functions \( f\colon \sym \rightarrow H(R^{\lambda})\otimes H(S) \) satisfying \( f(xhg) = R^{\lambda}(x)S(h)f(g) \) for all \( x\in S(n) \),  \( h\in \sym[n] \) and  \( g\in \sym \), and such that \( f \) lies in \( \ell^2(L\backslash \sym) \).

	The last condition means the following: \( \|f(g)\| \) is constant along right cosets of \( L \), so we can consider the function \( \bar{f}\colon L \backslash\sym \rightarrow \mathbb{C} \) defined by \( \bar{f}(Lg) = \|f(g)\|\) and we want
	\[
	\sum_{\sigma}\| \bar{f}(\sigma)\|^2 <\infty,
	\]
	where \( \sigma \) runs along all right cosets \( L\backslash \sym \).

	The space of \( K_{n}^{\alpha} \)-invariant functions in \( H(T) \) is exactly the space of functions with support in \( L \). Indeed, let \( f \) be \( K_{n}^{\alpha} \)-invariant. Then \( \bar{f} \) is constant along \( K_{n}^{\alpha} \)-orbits, so if an orbit \( gK_{n}^{\alpha} \) intersects with infinitely many right cosets \( Lx \), then \( f(g) = 0 \).

	The right coset \( Lx \) consists of all \( y \) such that \( x^{-1}([n]) = y^{-1}([n]) \), while an orbit \( gK_{n}^{\alpha} \) consists of all \( y \) such that \( y(k) = g(k), k\in [n] \) and  \( y(\alpha_i) = g(\alpha_i) \) for any \( i\in I \).

	So we need to show that if \( g \) does not lie in \( L \), its orbit \( gK_{n}^{\alpha} \) intersects with infinitely many cosets \( Lx \). Note that \( gK_{n}^{\alpha} \) intersects with \( Lx \) if and only if \( x^{-1}([n]) \) lies the set \( K_n^{\alpha}g^{-1}([n]) = \{xg^{-1}([n]) \mid x\in K_n^{\alpha}\} \). But when \( g\notin L \) there exists a point \( a\in g^{-1}([n])\setminus[n] \), therefore the set \( K_n^{\alpha}g^{-1}([n])\) is infinite and \( gK_{n}^{\alpha} \) intersects with infinitely many \( Lx \).

	So, if \( f \) is \( K_{n}^{\alpha} \)-invariant, \( f(g) = 0 \) for any \( g\notin L \).

	All such functions \( f \) are determined by their value at identity \( f(1) \), and, moreover, \( f(1) \) must lie in \( (H(R^{\lambda}\otimes S))^{K_{n}^{\alpha}}= H(R^{\lambda})\otimes \xi \). So, the subspace of  \mbox{\( K_{n}^{\alpha} \)-invariant} functions is \( S(n) \)-irreducible and it is cyclic, so \( T \) itself is irreducible.

	Since \( T \) is irreducible with \( H^{\alpha}_{n}(T)\neq 0 \), it is \( \alpha \)-admissible.
\end{proof}

Note that if \( \beta \) is a partition such that \( [\beta]\succcurlyeq [\alpha] \), then \( T(\alpha, X, \lambda, G) \) is clearly \( \beta \)-admissible.

\begin{theorem}[cf. {\cite[Theorem~5.9]{Nessonov}}]\label{Theorem:Admissible1}
	Let \( T \) be some irreducible \( \beta \)-admissible representation of \( \sym \). Then there exists a finite number \( n \), a Young diagram \( \lambda \) of the size \( n \), a partition \( \alpha \in \mathcal{P}^{I} \) such that \( [\alpha]\preccurlyeq[\beta] \) and a Gram matrix \( G\in \mathcal{G}^{I} \) such that \( T \) is equivalent to \( T(\alpha, n, \lambda, G) \).
\end{theorem}

\begin{remark}
	On the first glance, it seems possible to reformulate the statement of this theorem in the following way: if \( T \) is an irreducible \( \beta \)-admissible representation of \( \sym \), then there exists a finite set \( X \), an irreducible representation \( R \) of \( S(X) \) and a \( \beta^{(X)} \)-spherical representation \( S \) of \( \sym[X] \) such that
	\[
	T = \Ind^{\sym}_{S(X)\sym[X]} (R\otimes S).
	\]
	But this reformulation will not be correct, as is shown in the example below.
\end{remark}
\begin{example}
	Let \( \beta \) be the partition of \( \mathbb{N} \) onto the sets of even and odd numbers, i.e., \( I = \{1,2\} \),  \( \beta_1 = \{1,3,5,\ldots\} \), \( \beta_2 = \{2,4,6\ldots\} \). Now, let us change \( \beta \) a bit. Write
	\begin{gather*}
	\beta'_1 = \beta_1\setminus\{1\} =  \{3,5,7,\ldots\},\\
	\beta'_2 = \beta_2\cup \{1\} =  \{1,2,4,6,\ldots\}.
	\end{gather*}
	Any \( \beta' \)-spherical representation is uniquely determined by a real number \( a\in [0,1]\) (see Example~\ref{Example:Spherical}). The representation \( S(\beta', a) \) is \( \beta \)-spherical when \( a\neq 0 \) (more on equivalence between \( S(\alpha, G)  \) and \( S(\alpha', G') \) for different \( \alpha, \alpha' \) will be given in Section~\ref{Section:Equivalence}).

	Now take \( T = S(\beta', 0) \). It is \( \beta \)-spherical, but is not equivalent to \( \Ind^{\sym}_{S(X)\sym[X]} (R\otimes S) \) for any choice of \( X, R, S \), because the latter representation does not have any non-zero \( K^{\beta'} \)-invariant vectors.
\end{example}
\begin{definition}[{\cite[Definition~5.2]{Nessonov}}]
	The \emph{depth} of an irreducible \( \beta \)-admissible representation \( T \) is the smallest number \( n \) such that there  exists a partition \( \beta' \) equivalent to \( \beta \) such that \( H_n^{\beta'}(T) \) is not zero.
\end{definition}

\begin{proposition}\label{Proposition:Depth}
	The depth of \( T(\alpha, n, \lambda, G) \) is \( n \).
\end{proposition}
\begin{proof}
	The proof repeats the argument of Theorem~\ref{Theorem:IrrAndAdm}.

	If \( f\colon \sym \rightarrow H(R^{\lambda})\otimes H(S^{(\alpha)}, G) \) is \( K_k^{\beta} \)-invariant for some \( \beta\sim \alpha \), then \( f(g) = 0 \) for any \( g \) such that orbit \( gK_k^{\beta} \) intersects with infinitely many cosets \( Lx \). If \( k<n \) it happens for any \( g\in \sym \). Indeed, recall that the orbit \( gK_k^{\beta} \) intersects with the coset \( Lx \) if and only if \( x^{-1}([n]) \) lies in \( K_k^{\beta}g^{-1}([n]) \). But when \( k<n \) the set \( K_k^{\beta}g^{-1}([n]) \) is infinite for any \( g\in \sym \), so the only \( K^{\beta}_k \)-invariant vector \( f \) is \( 0 \).
\end{proof}
Recall that for any two finite sets \( X\subset Y \) we introduced an idempotent \( \epsilon^{\beta}_{Y,X} \) such that the orthogonal projection \( H^{\beta}_Y\rightarrow H^{\beta}_X \) is given by \( \T_{Y}(\epsilon^{\beta}_{Y,X}) \).
\begin{lemma}
	Let \( \epsilon \) be some idempotent in \( \G^{\beta}(Y) \), i.e., an element
	\[
	\epsilon = \begin{bmatrix}
		1_X & (B_j)^t\\
		(B_i) & (0)
	\end{bmatrix}
	\]
	for some \( X\subset Y \) and \( B_i\subset Y\setminus X \), \( i\in I \). Then \( \T^{\beta}_{Y}(\epsilon) \) gives the orthogonal projection from the space \( H^{\beta}_{Y} \) to the space \( H^{\alpha}_{X} \) for any partition \( \alpha \) satisfying:
	\[
	\alpha_i^{(X)} = \beta_i^{(Y)}\cup B_i.
	\]
\end{lemma}
\begin{proof}
	Since \( H^{\alpha}_{Y} = H^{\beta}_{Y} \), this projection is given by
	\[
	\T^{\alpha}_{Y}(\epsilon^{\alpha}_{Y,X}) = P^{\alpha}_Y T(w^{\alpha}_{Y,X})P^{\alpha}_Y = P^{\beta}_Y T(w^{\alpha}_{Y,X})P^{\beta}_Y = \T^{\beta}_Y(\epsilon).
	\qedhere
	\]
\end{proof}

\begin{proof}[Proof of the theorem]
	Suppose that \( n \) is the depth of \( T \) and \( H(T) \) possesses non-zero \( K_n^{\beta'} \)-invariant vectors for some \( \beta'\sim \beta \).

	If \( n = 0 \), then \( T \) is \( \beta' \)-spherical, and everything is proven. So now suppose \( n > 0 \).

	If \( X\subsetneq [n] \) and \( \epsilon\in \G^{\beta'}(n) \) is an idempotent of the form
	\[
	\epsilon = \begin{bmatrix}
		1_X & (B_j)^t\\
		(B_i) & (0)
	\end{bmatrix},
	\]
	\( \T_n(\epsilon) \) is a projection on the space \( H_{X}^{\tilde{\beta}} \), where \( \tilde{\beta }^{(X)} = \beta'^{(n)}\sqcup B_i\). Let \( |X| = k<n \) and \( g \) be any permutation from \( \sym \) that maps \( X \) to \( [k] \). Then \( T(g^{-1})H_{X}^{\tilde{\beta}} = H_k^{g(\tilde{\beta})} = 0 \) since \( g(\tilde{\beta})\sim \beta \).

	So \( \T_n(\epsilon) \) acts on \( H^{\beta'}_n \) by \( 0 \) for any idempotent \( 1\neq \epsilon\in \G^{\beta'}(n) \).  Let \( \mathfrak{I} \) be the two-sided ideal generated by all such \( \epsilon \). We claim that \( \mathfrak{I} \) consists of all \( \g \)
	\[
	\g = \begin{bmatrix}
		A& (C_j)\\
		(B_i) & (d_{ij})
	\end{bmatrix}
	\]
	such that \( \dom(A)\neq [n] \).

	Indeed, suppose that \( \g \) be as above. Then take \( X = \dom(A) \) and
	\[
	\epsilon = \begin{bmatrix}
		1_X & (B_j)^t\\
		(B_i) & (0)
	\end{bmatrix},
	\]
	where \( B_i \) are the same as in \( \g \). Then
	\begin{gather}
	\g\cdot\epsilon = \begin{bmatrix}
		A\cdot 1_X & (AB_j^t + C_j)\\
		(B_i\cdot 1_X + B_i) & (B_iB_j^t + d_{ij})
	\end{bmatrix} =
	\begin{bmatrix}
		A & (C_j)\\
		(B_i) & (d_{ij})
	\end{bmatrix} = \g
	\end{gather}
	lies in the ideal \(\mathfrak{I}\).

	So the representation \( \T^{\beta'}_{n} \) is uniquely determined by its restriction to the subset \( \G_0 = \G^{\beta'}(n)\setminus\mathfrak{I} \) that consists of all elements
	\[
	\g = \begin{bmatrix}
		A & (0)\\
		(0) & (d_{ij})
	\end{bmatrix},
	\] where \( A \) is a permutation \( A\in S(n) \). The set \( \G_0 \) is a sub-semigroup of \( \G^{\beta'}(n) \) and is isomorphic to the direct product of \( S(n) \) and \( \G^{\beta'^{(n)}} \). Its irreducible representations are exactly \( R^{\lambda}\otimes\chi \), where \( \chi \) is some character of \( \G^{\beta'^{(n)}} \) and uniquely corresponds to some \( \beta'^{(n)} \)-spherical representation of \( \sym[n] \).
\end{proof}

Note that the proof of this theorem gives one an algorithm on how to find all the parameters \( \alpha, n, \lambda, G \) given an irreducible \( \beta \)-admissible representation \( T \).

\section{Unitary equivalence for \(\beta\)-admissible representations}\label{Section:Equivalence}

In this section we find conditions for when two representations \(T_1 =  T(\alpha, n, \lambda, G) \) and \( T_2 = T(\beta, m, \mu, G') \) are unitary equivalent.

\begin{proposition}\label{Proposition:AdmisOrder}
	Representation \( T = T(\alpha, n, \lambda, G) \) is \( \beta \)-admissible if and only if \( [\beta]\succcurlyeq[\alpha] \).
\end{proposition}

\begin{proof}
	Suppose that \( \alpha = (\alpha_i)_{i\in I} \), \( \beta = (\beta_j)_{j\in J} \) and \( [\beta]\not\succcurlyeq [\alpha]\). It means that for any \( N \) there exist \( j\in J \) such that \( \beta^{(N)}_{j} \) intersects with at least two different \( \alpha^{(N)}_{k}, \alpha^{(N)}_{l} \).

	Firstly, assume \( n = 0 \) and \( T = S(\alpha, G) \) is a spherical representation with spherical vector \( \xi = \xi_1\otimes\cdots \).

	Let \( \eta\in H(T) \) be unit \( K_n^{\beta} \)-invariant vector. Since \( \xi \) is cyclic, for any \( \varepsilon>0 \) there exists
	\[
	\eta_{\varepsilon} = \sum_{i=1}^m c_iT(g_i)\xi, \quad \|\eta-\eta_{\varepsilon}\|<\varepsilon/2
	\]
	and one can easily check that
	\[
	\|T(k)\eta_{\varepsilon} - \eta_{\varepsilon}\|<\varepsilon
	\]
	holds.
	For any \( k\in K_n^{\beta} \)
	\begin{gather*}
		\|{T(k)\eta_{\varepsilon}}\|^2= \|{\eta_{\varepsilon}}\|^2 = \sum_{i,j = 1}^m c_i\bar{c_j}(T(g_i)\xi, T(g_j)\xi) = \sum c_i\bar{c_j} \varphi(g_j^{-1}g_i)>1-\varepsilon/2
	\end{gather*}
	since \( \eta \) is unit. In the similar manner,
	\begin{gather*}
		\left( T(k)\eta_{\varepsilon}, \eta_{\varepsilon} \right) = \sum c_i\bar{c_j} \varphi(g_j^{-1}kg_i).
	\end{gather*}
	It is true for all \( k\in K_n^{\beta} \). We may take \( N>n \) such that all \( g_i \) lie in \( S(N) \). Then for any \( k\in K_N^{\beta} \)
	\[
	\varepsilon > \|T(k)\eta_{\varepsilon} - \eta_{\varepsilon}\| = \sum_{i,j = 1}^m (c_i\bar{c_j}\varphi(g_j^{-1}g_i)) (2-\varphi(k) - \bar{\varphi(k)})> (1-\varepsilon/2)(2-\varphi(k) - \bar{\varphi(k)}).
	\]
	Now, since \( [\beta]\not\succcurlyeq [\alpha] \) we have an infinite collection of pairwise different numbers \( x_1, y_1, x_2, y_2, \ldots >N\) such that for any \( s \) elements \( x_s, y_s \) lie in the same part \( \beta_j \), but in different parts \( \alpha_k, \alpha_l \). If we now take \( k_r = (x_1, y_2)\cdots(x_r,y_r) \in K_N^{\beta}\), the sequence \( \varphi(k_r) \) converges to \( 0 \), so for \( r \) large enough
	\[
	(2-\varphi(k_r) - \bar{\varphi(k_r)}) > \frac{\varepsilon }{2-\varepsilon}
	\]
	and we have a contradiction.

	Now suppose \( T = T(\alpha, n, \lambda, G) \) and \( f\neq 0 \) is \( K_N^{\beta} \)-invariant. We may assume that \( N>n \). Using the same argument we used in the proof of Theorem~\ref{Theorem:Admissible1} and Proposition~\ref{Proposition:Depth} we obtain that the support of \( f \) consists only of those \( g \) such that \( g^{-1}([n]) \subset [N] \). Moreover, \( f \) is determined by its values on the elements \( g_X \), where \( X\subset [N] \) is any subset of size \( n \) and \( g_X\in S(N) \) are some permutations mapping \( X \) onto \( [n] \). Since any \( k\in K^{\beta}_N \) commutes with all \( g_X \) and lies in \( \sym[n] \), all af the values \( g_X \) must lie in \( \in H(R^{\lambda})\otimes H( S(\alpha^{(n)}, G ) )^{K_N^{\beta}} \), so for \( f \) to be non-zero, \( S(\alpha^{(n)}, G) \) must be \( \beta^{(n)} \)-admissible, i.e., \( [\beta] \preccurlyeq [\alpha] \).
\end{proof}
It follows that if two representations \( T_1 = T(\alpha, n, \lambda, G) \) and \( T_2 = T(\beta, m, \mu, G') \) are equivalent, then \( \alpha \sim \beta \) and \( m = n \), since the depth of \( \alpha \)-admissible representation does not depend on the choice of representative \( \alpha \) in the equivalence class.

\begin{proposition}~\label{Proposition:G-closeness}
	Let \( \alpha\in \mathcal{P}^{I} \) be some partition of \( \mathbb{N} \), \( G = (g_{ij}) \in \mathcal{G}^{I} \) be some Gram matrix, and \( S = S(\alpha, G) \) be the corresponding spherical representation. Recall that it is a closed subspace of the Hilbert space \( V^{\otimes\infty} (\xi)\).
	Now suppose that \( \beta \sim \alpha \) is obtained from \( \alpha \) by moving some number \( k \) from \( \alpha_i \) to \( \alpha_j \). In other words, there are \( i\neq j \in I \) and \( k\in \alpha_i \) such that \( \beta_i = \alpha_i\setminus\{k\} \), \( \beta_j = \alpha_j\cup \{k \} \) and \( \beta_{i'} = \alpha_{i'} \) for \( i'\neq i,j \).

	Then \( S \) has a non-zero \( K^{\beta} \)-invariant vector \( \eta \) if and only if \( g_{ij}\neq 0 \). In this case we say that \( \beta \) is obtained from \( \alpha \) by a \emph{\( G \)-permissible movement}.
\end{proposition}

\begin{proof}
	Note that for any \( \beta\sim \alpha \) the space \( V^{\otimes}(\xi) \) has exactly one (up to multiplication by scalar) \( K^{\beta} \)-invariant vector \( \eta \), namely	\(	\eta = \eta_1\otimes\ldots \),	where \( \eta_k = \ v_i \) when \( k\in \beta_i \).

	Now, if \( g_{ij} = 0 \), then for any \( x\in \sym \) vector \( \eta \) is orthogonal to \( T(x)\xi \), so lies in the orthogonal complement to \( H(S) \) in \( V^{\otimes \infty } (\xi) \). It follows that \( H(S) \) does not have any non-zero \( K^{\beta} \)-invariant vectors. In particular, \( S(\alpha, G) \) is not equivalent to \( S(\beta, H) \) for any choice of \( H \).

	Now suppose \( g_{ji}\neq 0 \). Let \( \alpha_j = (a_j^1<a_j^2<\ldots) \) and take
	\[
	\eta_n = \frac{1}{ng_{ji}} \sum_{s = 1}^{n} T(k, a_j^s)\xi.
	\]
	Then
	\begin{gather*}
		(\eta_n, \eta_n) = \frac{1}{n^2|g_{ji}|^2} \sum_{s\neq t = 1}^n \langle T(k, a_j^s, a_j^t)\xi, \xi \rangle + \frac{n}{n^2|g_{ji}|^2} = \frac{n(n-1)|g_{ji}|^2}{n^2|g_{ji}|^2} + \frac{n}{n^2|g_{ji}|^2};\\
		(\eta_n, \eta) = \frac{1}{ng_{ji}} \sum_{s = 1}^n (T(k,a_j^s)\xi, \eta) = \frac{1}{n g_{ji}} n g_{ji} = 1;\\
		\| \eta - \eta_n \|^2 = 1 + \frac{n(n-1)}{n^2} + \frac{1}{n |g_{ji}|^2} - 2 \rightarrow 0,
	\end{gather*}
	So \( \eta \) lies in \( H(S) \) and \( S \) is equivalent to \( S(\beta, G) \) for the same Gram matrix \( G \).
\end{proof}

\begin{example}
	Suppose that \( |I| = 2 \), \( \alpha\) is a partition onto even and odd numbers and \( \beta \sim \alpha \). Recall that \( \alpha \)-spherical representations are parametrized by \( g\in [0,1] \). From the proposition it follows that \( S(\alpha, g) \) is equivalent to \( S(\beta, g) \) whenever \( g\neq 0 \). Representations \( S(\alpha, 0) \) and \( S(\beta, 0) \) are equivalent if and only if \( \beta \) is obtained from \( \alpha \) by some finite permutation of \( \mathbb{N} \).
\end{example}

\begin{definition}
	Let \( \alpha \) be some partition of \( \mathbb{N} \) and suppose that \( \beta\sim \alpha \) can be obtained from \( \alpha \) by some finite permutation and by composition of finite number of \( G \)-permissible movements. We will say that \( \alpha \) and \( \beta \) are \emph{\( G \)-close}.
\end{definition}

\begin{definition}
	For a partition \( \alpha \in \mathcal{P}^{I} \) and a Gram matrix \( G \in \mathcal{G}^{I} \) we define a graph \( \Gamma(G) \) as in the proof of the third part of Theorem~\ref{Theorem:Spherical}. I.e, its vertices are indexed by \( I \), and two vertices \( i \neq j \) are connected by an edge if \( g_{ij}\neq 0 \). Then connected components \( c \) of \( \Gamma(G) \) define a splitting of \( I \) into subsets \( I = \bigsqcup_{c\in c} I_{c} \).

	For any \( N\geqslant 0 \), \( \alpha\in \mathcal{P}^{I} \) and \( G\in \mathcal{G}^{I} \), we define a collection of numbers \( l(\alpha, G)^{(N)} = \left(l(\alpha, G)^{(N)}_{c}\right)_{c\in C}\), where \( l(\alpha, G)^{(N)}_{c} \) denotes the number of points \( k\in [N] \) such that \( k\in \alpha_i \) and \( i\in I_{c} \).
\end{definition}

 Then two partitions \( \alpha, \beta \) are \( G \)-close if and only if for all \( N \) large enough we have \( l(\alpha, G)^{(N)} = l(\beta, G)^{(N)} \).

\begin{proposition}[cf. {\cite[Theorem~3.7]{Nessonov}}]
	Let \( \alpha\in \mathcal{P}^{I} \), \( G\in \mathcal{G}^{I} \) and \( S(\alpha, G) \) be the corresponding representation of \( \sym \) lying in the space \( V^{\otimes \infty} (\xi) \).

	Let \( C \) be the set of connected components of the graph \( \Gamma(G) \) and for any \( c\in C \) we set \( V_c \) to be the closed subspace of \( V \) generated by all \( v_i \) for \( i\in c \). Then \( V \) splits into orthogonal partition
	\[
	V = \bigoplus_{c\in C} V_c
	\]
	and, in turn, \( G \) splits into direct sum of Gram matrices \( G_c\in \mathcal{G}^{c} \).

	Now let \( \beta \) be the partition obtained from \( \alpha \) by uniting all parts lying in the same component. In other words, \( \beta\in \mathcal{P}^{C} \) and
	\[
	\beta_c = \bigsqcup_{i\in c} \alpha_i.
	\]
	Then the representation \( S(\alpha, G) \) is equivalent to the induced representation
	\[
	\Ind^{\sym}_{K^{\beta}} \bigotimes_{c\in C} S(\rest{\alpha}{\beta_c}, G_c),
	\]
	where \( \rest{\alpha}{\beta_c}\in \mathcal{P}^{c}(\beta_c) \) is the partition \( \beta_c = \bigsqcup_{i\in c} \alpha_i \) of the set \( \beta_c \), and the representation \( \bigotimes S(\rest{\alpha}{\beta_c}, G_c) \) is defined as follows. Its space is
	\[
	\bigotimes_{c\in C} (H(S(\rest{\alpha}{\beta_c}, G_c)), \xi^c),
	\]
	where \( \xi^c \) is the spherical vector of the corresponding representation. The action of \( K^{\beta} = \sum_{c\in C} S(\beta_c) \) is defined naturally.
\end{proposition}

\begin{proof}
	Proposition~\ref{Proposition:G-closeness} gives us a description of the space \( H \) of \( S(\alpha, G) \). For any \( \alpha' \sim \alpha \) the corresponding \( K^{\alpha'} \)-invariant vector \( \xi' \) lies in \( H \) when \( \alpha' \) is \( G \)-close to \( \alpha \) and is orthogonal to \( H \) otherwise. It follows that \( H \) has a dense subspace
	\[
	\bigcup_{n\in \mathbb{N}} V_{c_1}\otimes V_{c_2}\otimes \cdots \otimes V_{c_n}\otimes \xi_{n+1}\otimes\cdots,
	\]
	where for any connected component \( c \)  there are exactly \( l(\alpha, G)_{c}^{(n)} \) indexes equal to \( c \) among \( c_1,\ldots,c_n \).

	Now, for any \( k\in \mathbb{N} \) write \( c(k) \) for the connected component such that \( k\in \alpha_i\subset \beta_{c(k)} \). Then the space \( H' \) that is the closure of
	\[
	\bigcup_{n\in \mathbb{N}} V_{c(1)}\otimes V_{c(2)}\otimes \cdots \otimes V_{c(n)}\otimes \xi_{n+1}\otimes\cdots
	\]
	is \( K^{\beta} \)-invariant and equivalent to the representation \( \bigotimes_{c\in C} S(\rest{\alpha}{\beta_c}, G_c) \).

	From this the equivalence between \( S(\alpha, G) \) and \(
	\Ind\bigotimes S(\rest{\alpha}{\beta_c}, G_c),
	\) becomes clear.
\end{proof}

\begin{theorem}\label{Theorem:SphEquiv}
	Two representations \( S(\alpha, G) \) and \( S(\beta, G') \) are equivalent if and only if all of the following conditions are satisfied
	\begin{enumerate}[(i)]
		\item \( \alpha\sim \beta \). We will assume that \( \alpha \) and \( \beta \) are indexed by the same set \( I \).
		\item \( G' = DGD^{-1} \) for some diagonal matrix \( D \) with \( |d_i| = 1 \). In particular, they define the same subsets \( I_{c} \subset I \) and the same notion of \( G \)-closeness.
		\item \( \alpha \) and \( \beta \) are \( G \)-close.
	\end{enumerate}
\end{theorem}

\begin{proof}
	The ``if'' part of the proposition is clear. Now suppose that \( S(\alpha, G) \) and \( S(\beta, G') \) are equivalent. It immediately follows that \( \alpha \) and \( \beta \) must be equivalent. Suppose that \( N \) is such that \( \alpha^{(N)} = \beta^{(N)} \). Then for any \( x\in \sym[N] \) the spherical functions \( \varphi^{\beta}(x) = \varphi^{\alpha}(x) \), so \( G\ = DGD^{-1} \).

	Now, from the previous arguments we see that the subspaces \( H(S(\alpha, G)) \) and \( H(S(\beta, G)) \) of \( V^{\otimes\infty}(\xi) \) coincide whenever \( \alpha \) and \( \beta \) are \( G \)-close and orthogonal otherwise. In the latter case the representations are not equivalent since in \(V^{\otimes\infty}(\xi) \) there is only one (up to multiplication) \( K^{\beta} \)-invariant vector.
\end{proof}
We are ready to prove the following
\begin{theorem}[cf. {\cite[Theorem~6.11]{Nessonov}}]
	Two representations \( T_1 = T(\alpha, n, \lambda, G) \) and \( T_2 = T(\beta, m, \mu, G') \) are equivalent if and only if
	\begin{enumerate}[(i)]
		\item \( \alpha\sim\beta \);
		\item \( n = m \);
		\item \( \lambda = \mu \);
		\item \( G' = DGD^{-1} \);
		\item partitions \( \alpha^{(n)} \) and \( \beta^{(n)} \) of \( \bar{[n]} \) are \( G\)-close.
	\end{enumerate}
\end{theorem}
\begin{proof}
	We have already seen that if \( T_1 \) and \( T_2 \) are equivalent, then \( \alpha\sim \beta \) and \( n = m \). The space \( H(T_1) \) must contain \( K_{n}^{\beta} \)-invariant vectors \( f \) uniquely determined by \( f(1)\in H(R^{\lambda})\otimes H(S(\alpha^{(n), G})^{K_{n}^{\beta}}) \), so \( \alpha^{(n)} \) and \( \beta^{(n)} \) must be \( G \)-close and \( G' = DGD^{-1} \) must hold.
\end{proof}

\section{Representations admissible with respect to two partitions}\label{Section:Intersection}
The classification results also give us the conditions under which an irreducible unitary representation \( T \) is \( \alpha \)- and \( \beta \)-admissible for different partitions \( \alpha, \beta \).
We will denote by \( \IrrAdm(\alpha) \) the set of equivalence classes of irreducible \( \alpha \)-admissible representations.

If \( T \) is \( \alpha \)-admissible, then \( T = T(\gamma, n, \lambda, G) \) for some \( [\gamma]\preccurlyeq [\alpha] \). It is also \( \beta\)-admissible if and only if \( [\gamma]\preccurlyeq [\beta] \).

\begin{definition}
	Let \( \alpha, \beta\in \mathcal{P} \). We say that \( \gamma \) is the \emph{infimum} of \( \alpha \) and \( \beta \) if \( \gamma\preccurlyeq \alpha, \beta\) and \( \gamma'\preccurlyeq \gamma \) for any other \( \gamma'\preccurlyeq \alpha, \beta \). We will denote this infimum by \( \alpha \wedge \beta \).
\end{definition}

\begin{proposition}
	The infimum \( \alpha \wedge \beta \) always exists.
\end{proposition}
\begin{proof}
	Let \( \alpha\in \mathcal{P}^{I} \), \( \beta \in \mathcal{P}^{J} \) and \( \Gamma \) be the graph with vertices indexed by \( I \), such that vertices \( i, i' \) are connected by an edge if there exists \( j\in J \) such that \( \beta_j \) intersects with both \( \alpha_i, \alpha_{i'} \). Let \( C \) be the set of connected components of \( \Gamma \), and put for any \( c\in C \)
	\[
	\gamma_c = \bigsqcup_{i\in c} \alpha_i.
	\]
	It is clear that \( \gamma = \alpha\wedge \beta\).
\end{proof}

\begin{definition}
We will say that an equivalence class \( [\gamma] \) is the \emph{infimum} of \( [\alpha] \) and \( [\beta] \) if \( [\gamma]\preccurlyeq [\alpha], [\beta]\) and \( [\gamma']\preccurlyeq [\gamma] \) for any other \( [\gamma']\preccurlyeq [\alpha], [\beta] \).
\end{definition}
 Unfortunately, the infimum of equivalence classes of partition may not exist. When it does we will denote it by \( [\alpha] \wedge [\beta] \).

\begin{example}~\label{Example:Infimum}
Here we will provide an example of two equivalence classes \( [\alpha], [\beta] \) such that the infimum does not exist. More explicitly, we will construct a series \( [\gamma(n)] \) of equivalence classes of partitions satisfying the following conditions:
\begin{enumerate}
\item \(
[\gamma(0)] \prec [\gamma(1)] \prec \cdots \preccurlyeq [\alpha], [\beta];
\)
\item  for any other \( [\gamma]\preccurlyeq [\alpha], [\beta] \) there exists \( n \) such that \( [\gamma] \prec [\gamma(n)] \).
\end{enumerate}

Let \( \alpha \in \mathcal{P}^{I} \) be a partition with countable index set \( I = \{1,2,\ldots\} \), suppose that \( \alpha_i = (a_i^1<a_i^2<\cdots) \), and suppose also that \( I \) is ordered in such a way that \( a_1^1<a_2^1<\ldots \). Note that this defines the bijection between \(\mathbb{N} \) and \( \mathbb{N} \times \mathbb{N} \) that sends \( a_i^j \) to the pair \( (i,j) \). It will be convenient for us to visualise the partition \( \alpha \) as follows (boxes of the same colour belong to the same part~\( \alpha_i \), the numbers \( a_i^j \) grow from left to right in every row and from top to bottom in the first column).

\[
\ytableausetup{nosmalltableaux}
\alpha =
\begin{ytableau}
	*(white)& *(white) &*(white) &*(white) & \none[\dots] \\
	*(mylightgray)& *(mylightgray) &*(mylightgray) &*(mylightgray) & \none[\dots] \\
	*(mygray)& *(mygray) &*(mygray) &*(mygray) & \none[\dots] \\
	\none[\vdots]&\none[\vdots]&\none[\vdots]&\none[\vdots]&\none
\end{ytableau}
\]
Now define \( \beta = (\beta_{i})_{i\in I} \) in the following way
\[
\beta_1 = \alpha_1\setminus \{a_1^1\}, \quad \beta_i = \left(\alpha_i \setminus\{a_i^1\}\right) \cup \{a_{i-1}^1\},\ i\neq 1.
\]
We can visualise \( \beta \) in the same manner.
\[
\ytableausetup{nosmalltableaux}
\beta =
 \begin{ytableau}
 	*(mylightgray)& *(white) &*(white) &*(white) & \none[\dots] \\
 	*(mygray)& *(mylightgray) &*(mylightgray) &*(mylightgray) & \none[\dots] \\
 	*(mydarkgray)& *(mygray) &*(mygray) &*(mygray) & \none[\dots] \\
 	\none[\vdots]&\none[\vdots]&\none[\vdots]&\none[\vdots]&\none
 \end{ytableau}
\]

For any \( n \geqslant 0 \) let us denote by \( \gamma(n) \in \mathcal{P}^{[n+1]} \) the partition of \( \mathbb{N} \) into \( n+1 \) parts defined as follows:
\[
 \gamma(n)_{i} = \alpha_i,\ i\leqslant n, \quad \gamma(n)_{n+1} = \mathbb{N}\setminus \bigcup_{i = 1}^{n}\alpha_i.
\]
For example, for small \( n \) we have
\begin{gather*}
\gamma(0) =
	\begin{ytableau}
		*(white)& *(white) &*(white) &*(white) & \none[\dots] \\
		*(white)& *(white) &*(white) &*(white) & \none[\dots] \\
		*(white)& *(white) &*(white) &*(white) & \none[\dots] \\
		*(white)& *(white) &*(white) &*(white) & \none[\dots] \\
		\none[\vdots]&\none[\vdots]&\none[\vdots]&\none[\vdots]&\none
	\end{ytableau},
\quad
\gamma(1) =
\begin{ytableau}
	*(white)& *(white) &*(white) &*(white) & \none[\dots] \\
	*(mylightgray)& *(mylightgray) &*(mylightgray) &*(mylightgray) & \none[\dots] \\
	*(mylightgray)& *(mylightgray) &*(mylightgray) &*(mylightgray) & \none[\dots] \\
	*(mylightgray)& *(mylightgray) &*(mylightgray) &*(mylightgray) & \none[\dots] \\
	\none[\vdots]&\none[\vdots]&\none[\vdots]&\none[\vdots]&\none
\end{ytableau},\quad
\gamma(2) =
\begin{ytableau}
	*(white)& *(white) &*(white) &*(white) & \none[\dots] \\
	*(mylightgray)& *(mylightgray) &*(mylightgray) &*(mylightgray) & \none[\dots] \\
	*(mygray)& *(mygray) &*(mygray) &*(mygray) & \none[\dots] \\
	*(mygray)& *(mygray) &*(mygray) &*(mygray) & \none[\dots] \\
	\none[\vdots]&\none[\vdots]&\none[\vdots]&\none[\vdots]&\none
\end{ytableau}.
\end{gather*}

One can note that \( \gamma(n)^{(a_n^1)} = \alpha^{(a_n^1)} \wedge \beta^{(a_n^1)} \).

Indeed, it is clear that the finest partition that is coarser then both \( \alpha \) and \( \beta \) is the trivial partition \( \gamma(0) \). Once you remove the first \( a_n^1 \) elements (i.e., the first \( n \) elements in the first column and some finite number of elements in other columns) the sets \( \alpha_i^{(a_n^1)} \) and \( \beta_i^{(a_n^1)} \) coincide for \( i = 1,\ldots, n \), and we have \( \alpha^{(a_n^1)} \wedge \beta^{(a_n^1)} = \gamma(n)^{(a_n^1)} \).

 So we constructed an increasing sequence of equivalence classes that are all coarser then \( [\alpha] \) and \( [\beta] \):
 \[
  [\gamma(1)]\prec[\gamma(2)]\prec \cdots \prec [\alpha], [\beta].
 \]
Now suppose that \( [\gamma]\preccurlyeq [\alpha], [\beta] \) is any equivalence class that is coarser then both \( [\alpha], [\beta] \). Then there exists a number \( N \) such that \( \gamma^{(N)}\preccurlyeq \alpha^{(N)}, \beta^{(N)} \).

 But the infimum \( \alpha^{(N)}\wedge \beta^{(N)} \) is \( \gamma(n)^{(N)} \) for \( n \) satisfying  \( a^1_n\leqslant N < a^1_{n+1} \), so we have \( [\gamma]\preccurlyeq [\gamma(n)]\prec [\gamma(n+1)] \).
\end{example}

The argument used in this example can be applied in general situation.

Let \( \alpha \), \( \beta \) be any partitions, and suppose that \( \{[\gamma(n)]\}_{n\in \mathbb{N}} \) is the sequence of equivalent classes of partitions defined by \( \gamma(n)^{(n)} = \alpha^{(n)} \wedge \beta^{(n)} \). Then \( \{[\gamma(n)]\}_{n\in \mathbb{N}} \) satisfies the conditions from the Example~\ref{Example:Infimum}:

\begin{enumerate}
	\item \(
	[\gamma(0)] \preccurlyeq [\gamma(1)] \preccurlyeq \cdots \preccurlyeq [\alpha], [\beta];
	\)
	\item  for any other \( [\gamma]\preccurlyeq [\alpha], [\beta] \) there exists \( n \) such that \( [\gamma] \preccurlyeq [\gamma(n)] \).
\end{enumerate}

It follows that the infimum exists if and only if the sequence \( [\gamma(n)] \)  stabilises. For example, it happens when both \( \alpha, \beta \) are partitions of \( \mathbb{N} \) into finitely many parts. Indeed, in this case we may take \( N \) large enough so that \( \alpha_i\cap \beta_j \subset [N] \) for any \( \alpha_i \) and \( \beta_j \) such that their intersection is finite. Then \( [\gamma(n)] = [\gamma(n+1)] \) for all \( n>N \).

\begin{proposition}
	Let \( \alpha\), \(\beta \) be two partitions and \( \gamma(n) \) be as above. Then
	\[
	\bigcup_{n\in \mathbb{N}} \IrrAdm(\gamma(n))
	\]
	is the set of all irreducible representations that are admissible with respect to both \( \alpha \) and \( \beta \).

	In particular, when the infimum \( [\gamma] = [\alpha] \wedge [\beta] \) exists, the set of all irreducible representations admissible with respect to both \( \alpha \) and \( \beta \) is \( \IrrAdm(\gamma) \).
\end{proposition}
The proof follows directly from the Proposition~\ref{Proposition:AdmisOrder} and the reasonings above.

Note that this set can be described in other terms.

Let \( K_n := K_n^{\gamma(n)} \). These groups form a decreasing sequence
\[
K_0 \supset K_1 \supset \cdots
\]
that satisfy the conditions~\ref{Definition:Topology1}~---~\ref{Definition:Topology3} from Definition~\ref{Definition:Topology}. So, we may define the \( \{K_n\} \)-topology in the usual way. The set of equivalence classes of irreducible representations continuous with respect to this topology is exactly \( \bigcup_{n\in \mathbb{N}} \IrrAdm(\gamma(n)) \).

It follows from the fact that \( K_N^{\gamma(n)} \) is sandwiched between \( K_n \) and \( K_N \) (depending on what is greater, \( n \) or \( N \)).

\printbibliography

\end{document}